\documentclass[12pt]{amsproc}
\usepackage[T1]{fontenc}

\usepackage{mathrsfs}

\usepackage{amssymb}
\usepackage{amsmath}
\usepackage{amsthm}
\usepackage{amsfonts}
\usepackage{fullpage}
\usepackage{enumerate}
\usepackage{tikz-cd} 
\usepackage{comment}
\usepackage{hyperref}
\usepackage{float}
\usepackage{verbatim}
\usepackage{enumitem}
\usepackage{fancyhdr}
\usepackage{todonotes}
\usepackage[foot]{amsaddr}
\usepackage{svg}
\usepackage{multirow, url}
\usepackage{textcomp}

\hypersetup{colorlinks}
\hypersetup{colorlinks=true, linkcolor=red, citecolor=blue, filecolor=blue, urlcolor=blue}

\newtheorem{thm}{Theorem}[section]
\newtheorem{cor}[thm]{Corollary}
\newtheorem{lem}[thm]{Lemma}
\newtheorem{prop}[thm]{Proposition}
\theoremstyle{definition}

\numberwithin{equation}{section} 
\numberwithin{table}{section}
\newtheorem{innercustomthm}{Theorem}
\newenvironment{customthm}[1]{\renewcommand\theinnercustomthm{#1}\innercustomthm}{\endinnercustomthm}

\newcommand{\nc}{\newcommand}
\newcommand{\rc}{\renewcommand}

\rc{\t}{\text}
\nc{\tb}{\textbf}
\nc{\mb}{\mathbb}
\nc{\tc}{\textcolor}
\nc{\mf}{\mathfrak}
\nc{\mc}{\mathcal}
\nc\on{\operatorname}

\rc{\AA}{\mathbb{A}}
\nc{\CC}{\mathbb{C}}
\nc{\EE}{\mathbb{E}}
\nc{\FF}{\mathbb{F}}
\nc{\II}{\mathbb{I}}
\nc{\PP}{\mathbb{P}}
\nc{\QQ}{\mathbb{Q}}
\nc{\RR}{\mathbb{R}}
\nc{\ZZ}{\mathbb{Z}}
\rc{\SS}{\mathbb{S}}

\rc{\d}{\on{d}}
\nc{\Law}{\on{Law}}
\nc{\Prob}{\mathbb{P}}
\nc{\Unif}{\mathbf{U}}
\nc{\Spec}{\on{Spec}}
\nc{\Proj}{\on{Proj}}

\nc{\Circle}{{\mathbb{S}^1}}
\nc{\sphere}{\mathbb{S}}
\nc{\Cech}{\text{\v{C}}}
\nc{\rr}{\tilde{r}}

\nc{\explainbox}[2]{
\vspace{3mm}
\noindent
\fbox{
    \parbox{\textwidth-6mm}
        {\vspace{1mm}
            \textbf{\underline{#1}} {#2} 
            \vspace{1mm}}
    }\vspace{3mm}}

\begin{document}

\title{Strange Random Topology of the Circle}
\author[U. Lim]{Uzu Lim}
\email{lims@maths.ox.ac.uk}
\address{Mathematical Institute, University of Oxford, Radcliffe Observatory, Andrew Wiles Building, Woodstock Rd, Oxford OX2 6GG}
\maketitle
\begin{abstract}
    We characterise high-dimensional topology arising from a random Cech complex constructed on the circle. We explicitly compute expected Euler characteristic curve, where we observe plateaus and limiting spikes. The plateaus correspond to odd-dimensional spheres, and the spikes correspond to bouquets of even-dimensional spheres. In particular, the bouquets have arbitrarily large Betti number as the sample size grows larger. By departing from the conventional practice of scaling down filtration radii as the sample size grows large, our findings indicate that the full breadth of filtration radii leads to interesting systematic behaviour that cannot be regarded as "topological noise".
\end{abstract}

\section{Introduction}

A conventional wisdom in topological data analysis says the following: if we construct a simplicial complex from a random sample drawn from a manifold, then the topology of the simplicial complex approximates the topology of the manifold. Indeed this is true if we scale down the connectivity radius smaller as the sample size grows larger, but what happens when the connectivity radius stays the same?

We study the strange random topology of the circle and characterise its high-dimensional topology. Our main tool is expected Euler characteristics. We find intervals of radii in which the random Cech complex constructed from the circle is homotopy equivalent to bouquets of spheres, with positive probabilities. Here, a bouquet of spheres is the wedge sum $\vee^a \sphere^b$. It was known that only $a=1$ is allowed if $b$ is odd, and all $a\ge 1$ are allowed if $b$ is even \cite{ncca}. The single odd sphere $\sphere^{2b+1}$ appears with high probability over long intervals of filtration radii. Bouquets of even sphere $\vee^a \sphere^{2b}$ appears with a smaller but positive probability over shrinking intervals of radii, with arbitrarily large $a$. Let's now describe the setup.

\vspace{3mm}

\explainbox{Setup.}{Define the circle $\sphere^1$ as the quotient space $\sphere^1 = [0,1]/\sim$, as the interval of length 1, glued along endpoints: $0 \sim 1$. A bouquet of spheres $\vee^a \sphere^k$ is defined as the wedge sum of $a$ copies of $\sphere^k$.\footnote{We take the convention that for a topological space $K$, we define the 0-th wedge sum $\vee^0 K = *$, the singleton point set, and the 1st wedge sum $\vee^1 K = K$ itself.} 

For a positive integer $n$, let $\mathbf X_n$ be the i.i.d.\footnote{Independently and identically distributed} sample of size $n$, drawn uniformly from $\sphere^1$. The Cech complex of filtration radius $\le r$ is denoted by $\Cech(\mathbf X_n, r)$. In doing this construction, we always use the intrinsic topology of the circle, i.e. \textit{the Cech complex is a nerve complex constructed from arcs}. 

We denote the expected Euler characteristic and expected Betti number as follows:
\[ \bar \chi(n, r) = \EE_{\mathbf X_n} \bigg[ \chi(\Cech(\mathbf X_n, r)) \bigg], \quad \bar b_k(n, r) = \EE_{\mathbf X_n} \bigg[ \dim H_k (\Cech(\mathbf X_n, r)) \bigg] \]
}

\begin{figure}\label{fig1}
    \includegraphics[scale = 0.8]{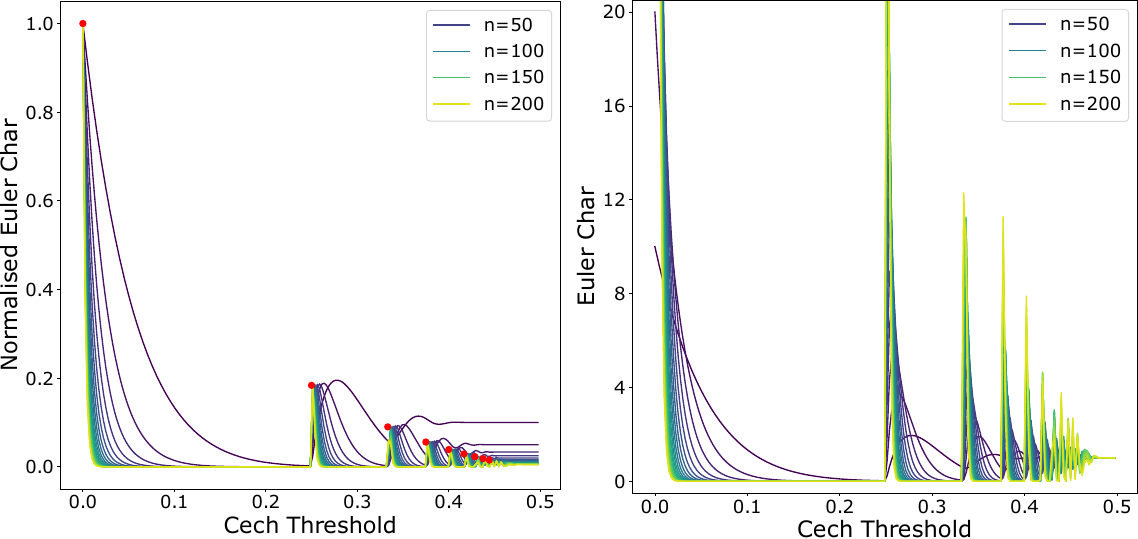}
    \caption{Left: Graphs of \textit{normalised} expected Euler characteristics, $y = n^{-1} \cdot \bar \chi(n, x)$, for $n \in \{10, 20, \ldots 200\}$ and $x \in [0,1]$. Right: Same as left, but we plot $y = \bar \chi(n, x)$, which are (un-normalised) expected Euler characteristics. Yellow curves correspond to larger $n$. Red circles are the limiting spikes, given by $\left( \frac{k}{2k+2}, \frac{(k/e)^k}{(k+1)!} \right)$ for all $k \ge 0$.}
\end{figure}

\explainbox{}{

\begin{customthm}{A}[Expected Euler Characteristic]\label{thmA}
    The following are true. \\

    (1) $\bar \chi(n, r)$ is a continuous piecewise-polynomial function in $r$, given explicitly as follows for $n>0$ and $r \in (0, 1)$:
    \[ \bar \chi\left( n, \frac{1-r}2 \right) = \sum_{k = 1}^{\lfloor 1/r \rfloor}  \binom nk (1 - kr)^{k-1} (kr)^{n-k} \]
    \vspace{3mm}
    
    (2) Normalised Euler charactreristics have the following peak values for all $k \ge 0$:
    \begin{align*}
        \lim_{n \rightarrow \infty} \max_{r \in I_k} \frac{\bar\chi (n,r)}n = \frac{(k/e)^k}{(k+1)!}
    \end{align*}
    where $I_k = \left( \frac{k}{2k+2}, \frac{k+1}{2k+4} \right)$.
    \vspace{3mm}

    (3) Let $k \ge 1$. Given $\epsilon > 0$, the following uniform bounds hold for all $r \in \left[ \frac{k}{2k+2}, \frac{k+1}{2k+4} \right)$ when $n$ is sufficiently large:
    \[ \frac{\bar \chi(n, r)}n - \epsilon \le  \frac{\bar b_{2k}(n, r)}n \le \frac{\bar \chi(n, r)}n \]
\end{customthm}
}

Theorem A allows us to plot \textit{exact values} of the expected Euler characteristic curves. The left side of Figure \ref{fig1} shows graphs of $f_n(r) = n^{-1} \cdot \bar\chi(n, r)$, which are \textit{normalised} versions of $\bar \chi$. We stress that these curves are not empirically obtained from a method like Monte Carlo; they are plots of the formula in Theorem A1. As $n$ becomes larger, $f_n(r)$ shows peaks that converge to a sequence of narrow spikes, as Theorem A2 predicts. The right side of the same figure shows the non-normalised graphs of $\bar\chi(n, r)$, where we see that the limiting peaks of Theorem A2 will peak into infinity as $n\rightarrow \infty$. 

Meanwhile all homotopy types arising from nerve complexes of circular arcs were completely classified in \cite{ncca}: they are either $\sphere^{2k+1}$ for some $k \ge 0$, or $\vee^a \sphere^{2k}$ for some $a, k \ge 0$. We observe that $\chi(\sphere^{2k+1}) = 0$, whereas $\chi(\vee^a \sphere^{2k}) = a+1$. Therefore in Figure \ref{fig1}, limiting spikes indicate contribution from the even-dimensional sphere bouquets $\vee^a \sphere^{2k}$ with large $a$, and the plateaus indicate contribution from the odd-dimensional spheres. Recalling that $\mathbf X_n$ is an i.i.d. sample of size $n$ drawn from $\sphere^1$, we have:

\explainbox{}{
\begin{customthm}{B}[Odd Spheres]\label{thmB}
    Let $k \ge 0$ be an integer, and also let $\epsilon, \delta >0$. Suppose that $|r - \nu_k| \le \tau_k-\epsilon$. Then for sufficiently large $n$, the following homotopy equivalence holds with probability at least $1-\delta$:
    \[ \Cech(\mathbf X_n, r) \simeq \SS^{2k+1} \]
    where
    \[ \nu_k = \frac{2k^2+4k+1}{4(k+1)(k+2)}, \quad \tau_k = \frac1{4(k+1)(k+2)} \]
\end{customthm}
}

\explainbox{}{
\begin{customthm}{C}[Even Spheres]\label{thmC}
    Let $k \ge 2$, $\eta \in (0,1)$. Suppose that $|r-\rho_{k,n}| \le \sigma_{k,\eta}/n$. Then for sufficiently large $n$, the following homotopy equivalence holds with probability at least $\eta \cdot k\omega_k$:
    \[ \Cech(\mathbf X_n, r) \simeq \vee^{a} \sphere^{2k-2}, \quad \text{for some } \frac{(1-\eta)\omega_k \cdot n}{2} \le a+1 \le \frac nk \]
    where
    \begin{align*}
        \rho_{k,n} = \frac{n(k+1)}{2k(n-1)}, \quad \sigma_{k, \eta} = \frac{(1-\eta)^3 (k\omega_k)^3}{320\sqrt{k+2}}, \quad \omega_k = \frac{(k-1)^{k-1}}{k! e^{k-1}}
    \end{align*}
\end{customthm}
}

\textbf{Remark 1.} In Theorem B, we note that $\nu_k = \frac12(\frac{k+1}{2k+4} + \frac{k}{2k+2})$ and $\tau_k = \frac12(\frac{k+1}{2k+4} - \frac{k}{2k+2})$, so that Theorem B covers most of each interval $r \in [\frac{k}{2k+2}, \frac{k+1}{2k+4}]$. Also to see Theorem C in action, one may simply set $\eta = 1/2$ to obtain results. 

\vspace{3mm}
\textbf{Remark 2.} Although all of the above results are proven for the Cech complex of circular arcs on the circle, similar behaviour are observed in the Rips complex constructed on the circle as well. Indeed, modifying Theorem B for the Rips complex immediately yields the following: for the Rips complex constructed from a finite random sample on a circle, all odd-dimensional spheres appear with positive persistence and probability approaching 1. Analogues of Theorems A and C for the Rips complex could not be immediately obtained with methods in this paper.

\vspace{5mm}
\textbf{Structure of the paper.} In Section 2 we prove Theorem A1, i.e. compute the expected Euler characteristic precisely. In Section 3 we prove Theorem A2 (Proposition \ref{main prop 1}), by analysing the limiting spikes of the expected Euler characteristics. In Section 4 we prove Theorem A3 and Theorem C, by using the classification of homotopy types arising from a nerve complex of circular arcs, thereby giving constraints on homotopy types and compute probabilistic bounds. In Section 5 we prove Theorem B, by using the classical method of stability of persistence diagram; this section works separately and doesn't use the Euler characteristic method.

Theorem C takes the most work to prove. It is a simplified version of Theorem \ref{elder C}, which has a few more parameters that can be tweaked to obtain similar variants of Theorem C. Theorem \ref{elder C} is obtained by combining three ingredients: Propositions \ref{main prop 1}, \ref{main prop 2}, and \ref{main prop 3}.

\vspace{5mm}
\textbf{Related works.} 

The classical result of Hausmann shows that the Vietoris-Rips complex constructed from the manifold with a small scale parameter recovers the homotopy type of the manifold \cite{hausmann}. Another classical result of Niyogi, Smale, Weinberger shows that if a Cech complex of small filtration radius is constructed from a finite random sample of a Euclidean submanifold, then the homotopy type of the manifold is recovered with high confidence \cite{nsw}.

Much work has been done for recovering topology of a manifold from its finite sample, when connectivity radius is scaled down with the sample size at a specific rate \cite{bobrowski1} \cite{vipond_boundary} \cite{kahle1} \cite{bobrowski2}. A central theme of this body of work is the existence of phase transitions when parameters controlling the scaling of connectivity radius are changed. For a comprehensive survey, see \cite{penrose_2003} and \cite{bobrowski_survey}. 

In comparison, the setting when connectivity radius is not scaled down with sample size is studied much less. Results on convergence of the topological quantities have been studied \cite{paik} \cite{euler_clt}, but not much attention has been devoted to analysing specific manifolds.

This paper builds on two important works that characterised the Vietoris-Rips and Cech complexes of subsets of the circle: \cite{ncca} and \cite{vr_circle}. Several variants of these ideas were studied, for ellipse \cite{vr_ellipse}, regular polygon \cite{vr_polygon}, and hypercube graph \cite{vr_hypercube}. Randomness in these systems were studied using dynamical systems in \cite{random_cyclic}. One key tool to further study the topology of Vietoris-Rips and Cech complexes arising from a manifold is metric thickening \cite{metric_thickening}. Using this tool, the Vietoris-Rips complex of the higher-dimensional sphere has been characterised up to small filtration radii \cite{vr_lim}.

\vspace{5mm}
{\footnotesize
\subsection*{Acknowledgements} ~\newline
The author is grateful to Henry Adams and Tadas Tem\v{c}inas for valuable discussions that led up to this paper. The author would also like to thank Vidit Nanda and Harald Oberhauser for their contributions during initial stages of this research.
\newline
\indent Uzu Lim is supported by the Korea Foundation for Advanced Studies. 
}

\section{Expected Euler characteristic}

In this section we compute the expected Euler characteristic precisely. We start with a simple calculation that also briefly considers the Vietoris-Rips complex, but soon after we only work with the Cech complex. Let $\on{VR}(\mathbf X_n, r)$ denote the Vietoris-Rips complex of threshold $r$. The following proposition reduces computation of expected values to the quantities $T_k$ and $Q_k$, defined below:
\begin{prop}
    For each $n>0$, let $\mathbf X_n$ be the iid sample drawn uniformly from $\mathbb S^1$. Then we have that:
    \begin{align*}
        \EE [\chi(\on{VR} (\mathbf X_n, r))] =& \sum_{k=1}^{n} (-1)^{k-1} \binom n{k} T_{k}(r) \\
        \EE [\chi(\Cech (\mathbf X_n, r))] =& 1 + \sum_{k=1}^{n} (-1)^{k} \binom n{k} Q_{k}(1 - 2r)
    \end{align*}
    where $T_k(r)$ is the probability that every pair of points in $\mathbf X_k$ are within distance $r$, and $Q_k(r)$ is the probability that open arcs of length $r$ centered at points of $\mathbf X_k$ cover $\mathbb S^1$. Expectation is taken over the iid sample $\mathbf X_n$.
\end{prop}
\begin{proof}
Denoting by $s_k(K)$ the number of $k$-simplices in a simplicial complex $K$, we have that:
\[ \EE [ s_k(\on{VR} (\mathbf X_n, r))] = \binom n{k} T_{k}(r) \]
and thus
\begin{align*}
    \EE [\chi(\on{VR} (\mathbf X_n, r))] = \sum_{k=0}^{n-1} (-1)^{k} \EE [ s_k(\on{VR} (\mathbf X_n, r))] = \sum_{k=1}^{n} (-1)^{k-1} \binom n{k} T_{k}(r)
\end{align*}
The relation for the Cech complex is derived in the same way, except we note the following: the probability that arcs of radius $r$ centered at points of $\mathbf X_k$ intersects nontrivially is equal to $1 - Q_k(1-2r)$. This is by De Morgan's Law: for any collection of sets $\{U_j \subseteq \mathbb S^1 \}_{j \in J}$, we have $\cap_{j \in J} U_j = \emptyset$ iff $\cup_{j \in J} U_j^{\on{c}} = \mathbb S^1$. In the case of circle (of circumference 1), complement of a closed arc of radius $r$ is an open arc of length $1-2r$. Applying this logic, we obtain:
\[ \EE [\chi(\Cech (\mathbf X_n, r))] = \sum_{k=1}^{n} (-1)^{k-1} \binom n{k} (1-Q_{k}(1-2r)) \]
which is easily seen to be the same as the asserted expression (note that $\sum_{k=1}^n (-1)^{k-1} \binom nk = 1$.)
\end{proof}

The $Q_k$ were computed by Stevens in 1939 \cite{stevens}. We reproduce the proof for completeness.
\begin{thm}[Stevens]
    If $k$ arcs of fixed length $a$ are independently, identically and uniformly sampled from the circle of circumference 1, then the probability that these arcs cover the circle is equal to the following:
    \[ Q_k(a) = \sum_{l = 0}^{\lfloor 1/a \rfloor} (-1)^l \binom kl (1-la)^{k-1} \]
\end{thm}
\begin{proof}
    The proof is an application of inclusion-exclusion principle. Consider the set $E = \{(x_1, \ldots x_k) | 0\le x_1 < \cdots < x_k < 1 \}$. For each collection of indices $J \subseteq \{1, \ldots k\}$, define $\bar E_J$ and $E_J$ as the following subsets of $E$:
    \begin{align*}
        E_J =& \{ (x_1, \ldots x_k) \in E | j \in J \iff x_{j+1} - x_j > a \} \\
        \bar E_J =& \{ (x_1, \ldots x_k) \in E | j \in J \implies x_{j+1} - x_j > a \} = \bigsqcup_{J' \supseteq J} E_{J'}
    \end{align*}
    By definition, we have $\on{Vol}(E_\emptyset) = Q_k(a)$. To compute it, we apply the inclusion-exclusion principle for the membership of each $E_J$ over $\bar E_{J'}$ whenever $J' \supseteq J$. Noting the relation $\sum_{l=1}^k (-1)^{l+1} \binom kl = 1$, we see that:
    \[ 1 = \sum_{J \subseteq \{1, \ldots k\}} \on{Vol}(E_J) = \on{Vol}(E_\emptyset) - \sum_{\emptyset \neq J \subseteq \{1, \ldots k\} } (-1)^{\# J} \on{Vol}(\bar E_J) \]
    Finally, if $l = \# J$ and $l \le \lfloor 1/a \rfloor$, then $\on{Vol}(\bar E_J) = (1-la)^{n-1}$. This is because demanding gap conditions $x_{i+1} - x_i > a$ at $l$ places is equivalent to sampling $n-1$ points from an interval of length $1-la$\footnote{This can be seen more precisely by considering the collection $E'$ of $(y_1, \ldots y_{k-1})$ defined by $y_i = x_{i+1} - x_i > 0$ and $\sum y_i \le 1$, and then considering the subset $E'_J$ defined by $y_i > a$ for $i \in J$. The quantity of interest is $\on{Vol}(E'_J) / \on{Vol}(E')$. Furthermore, the map $(y_1, \ldots y_{k-1}) \mapsto (y_1 - \mathbf{1}_{1 \in J}, \ldots y_{k-1} - \mathbf{1}_{k-1 \in J})$ isometrically maps $E_J'$ to $(1-la) \cdot E'$, so that $\on{Vol}(E'_J) = (1-la)^{k-1} \on{Vol}(E')$ due to the $(k-1)$-dimensional volume scaling. This is exactly the original claim.}. Meanwhile if $l > \lfloor 1/a \rfloor$, then we always have $\on{Vol}(\bar E_J) = 0$. Plugging these into the above equation, we get:
    \[ \on{Vol}(E_\emptyset) = 1 + \sum_{l=1}^{\lfloor 1/a\rfloor} (-1)^{l} \binom kl (1-la)^{n-1} \]
    as desired.
\end{proof}

We then get the following:

\begin{thm}[Theorem A]
    Expected Euler characteristic of random Cech complex on a circle of unit circumference obtained from $n$ points and filtration radius $(1-r)/2$ is:
    \[ \bar \chi\left( n, \frac{1-r}2 \right) = \sum_{k = 1}^{\lfloor 1/r \rfloor}  \binom nk (1 - kr)^{k-1} (kr)^{n-k} \]
    In particular, $\bar \chi(n, r)$ is a continuous piecewise-polynomial function in $r$.
\end{thm}
\begin{proof}
    Substituting the $Q_k$ expression in, we get:
    \begin{align*}
        \bar \chi\left( n, \frac{1-r}2 \right) =& 1 + \sum_{k=1}^{n} (-1)^{k} \binom n{k} Q_{k}(r) \\
        =& 1 + \sum_{l = 0}^{\lfloor 1/r \rfloor} \sum_{k=1}^{n} (-1)^{k+l} \binom n{k} \binom kl (1-r l)^{k-1} \\
        =& \sum_{l = 1}^{\lfloor 1/r \rfloor} \sum_{k=1}^{n} (-1)^{k+l} \binom n{k} \binom kl (1-r l)^{k-1}
    \end{align*}
    where we switched the order of summation in the second equality, and isolating the $l=0$ part cancels out the $1$ in the third equality. Noting that $\binom nk \binom kl = \binom nl \binom{n-l}{k-l}$, we further get:
    \begin{align*}
        \bar \chi\left( n, \frac{1-r}2 \right) =& \sum_{l = 1}^{\lfloor 1/r \rfloor} (-1)^{l} \binom nl (1-rl)^{-1} \sum_{k=l}^{n} \binom{n-l}{k-l} (r l - 1)^{k} \\
        =& \sum_{l = 1}^{\lfloor 1/r \rfloor} \binom nl (1-rl)^{l-1} \sum_{k=0}^{n-l} \binom{n-l}{k} (r l - 1)^{k} \\
        =& \sum_{l = 1}^{\lfloor 1/r \rfloor}  \binom nl (1 - rl)^{l-1} (rl)^{n-l}
    \end{align*}
\end{proof}

\section{Limit behaviour of Euler characteristic}

We prove a sequence of lemmas in this section to characterise the limiting spikes in Figure \ref{fig1}. The main idea is that only one summand in the expected Euler characteristic contributes mainly to the spike, and this is a polynomial term that can be studied with calculus. The main results of this section are Propositions \ref{spike prop} and \ref{main prop 1}. The two lemmas leading up to it are exercises in calculus that explain the specific situation of our expected Euler characteristic.

\begin{lem}
    For $a, b \ge 1$, the function $f(t) = t^a(1-t)^b$ satisfies the following:
    
    (a) In the range $0 \le t \le 1$, $f(t)$ achieves the unique maximum value at $t = a/(a+b)$:
    \[ \max_{0 \le t \le 1} f(t) = f\left( \frac{a}{a+b} \right) = \frac{a^a b^b}{(a+b)^{a+b}} \]
    Also, $f(t)$ is increasing on $t \in (0, a/(a+b))$ and decreasing on $t \in (a/(a+b), 1)$.
    
    (b) The following \textit{linear} lower bounds hold:
    \begin{align*}
        f(t) \ge & u\bigg( (a+b)vt  - av + 1 \bigg) \text{, when } 0 < t < \frac{a}{a+b} \\
        f(t) \ge & u\bigg( -(a+b)vt  + av + 1 \bigg) \text{, when } \frac{a}{a+b} < t < 1
    \end{align*}
    where
    \[u = \frac{a^a b^b}{(a+b)^{a+b}}, \quad v = \sqrt{\frac{a+b}{ab}} \]
    
    (c) For each $\lambda \in [0, 1]$, we have that:
    \[ \left| t - \frac a{a+b} \right| < \frac{(1-\lambda) \sqrt{ab}}{(a+b)^{3/2}} \implies t^a(1-t)^b > \lambda u \]
\end{lem}
\begin{proof}
    The first two derivatives are:
    \begin{align*}
        f'(t) =& \bigg( a - (a+b)t \bigg) t^{a-1} (1-t)^{b-1} \\
         f''(t) =& \bigg( (a+b)(a+b-1)t^2 + 2a(1-a-b)t + a(a-1) \bigg) t^{a-2} (1-t)^{b-2}
    \end{align*}
    The first derivative vanishes at $t \in \{ a/(a+b), 0, 1 \}$ and the second derivative vanishes at $t \in \{ t_0 \pm \eta_0, 0, 1 \}$ where
    \[ t_0 = \frac a{a+b}, \quad \eta_0 = \frac1{a+b}\sqrt{\frac{ab}{a+b-1}} > \frac{\sqrt{ab}}{(a+b)^{3/2}} = \eta_1 \]
    The first derivative is positive at $(0, a/(a+b))$ and negative at $(a/(a+b), 1)$. Thus the maximum at $t \in [0, 1]$ is given by:
    \[ f(t_0) = \frac{a^a b^b}{(a+b)^{a+b}} \]
    Thus
    \begin{align*}
        f(t) \ge \frac{f(t_0)}{\eta_1}(t-t_0) + f(t_0) \text{, when } 0 < t < t_0 \\
        f(t) \ge \frac{-f(t_0)}{\eta_1}(t-t_0) + f(t_0) \text{, when } t_0 < t < 1
    \end{align*}
    and
    \begin{align*}
        \pm \frac{f(t_0)}{\eta_1}(t-t_0) + f(t_0) =& \frac{a^a b^b}{(a+b)^{a+b}} \left( \pm \frac{(a+b)^{3/2}}{\sqrt{ab}} \left( t - \frac{a}{a+b} \right) + 1 \right)
    \end{align*}
    (c) follows from the linear bound of (b).
\end{proof}

\begin{lem}
    Let $m, n \ge 1$ be integers and define:
    \[ f_{m, n}(t) = \binom nm (mt)^{m-1} (1-mt)^{n-m} \]
    Then $f_{m, n}$ satisfies the following:
    
    (a) $f_{m,n}(t)$ is increasing when $0< t < t_0$ and decreasing when $t_0 < t < 1/m$ where $t_0 = \frac1{n-1}(1-\frac1m)$.
    
    (b) The maximum over $0 < t < 1/m$ is given by:
    \[ \max_{0 < mt < 1} f_{m,n}(t) = f_{m,n}(t_0) = \binom nm \frac{(m-1)^{m-1}(n-m)^{n-m}}{(n-1)^{n-1}} \]
    
    (c) For each $\lambda \in [0, 1]$, we have that:
    \[ |t-t_0| < \frac{(1-\lambda) \sqrt{(m-1)(n-m)}}{m (n-1)^{3/2}} \implies f_{m,n}(t) > \lambda f_{m, n}(t_0) \]
    
    (d) The normalised limit of maximum as $n \rightarrow \infty$ is given by:
    \begin{align*}
        \lim_{n \rightarrow \infty} \frac{\max_{0 < t < 1/m} f_{m,n}(t)}{n} = \frac{(m-1)^{m-1}}{m! e^{m-1}}
    \end{align*}
\end{lem}
\begin{proof}
    (a)-(c) follow from the previous lemma. For (d), we compute:
    \begin{align*}
         \lim_{n \rightarrow \infty} \frac{\max_{0 < t < 1/m } f_{m,n}(t) }n =& \frac{(m-1)^{m-1}}{m!} \lim_{n \rightarrow \infty} (n-1)(n-2) \cdots (n-m+1) \frac{(n-m)^{n-m}}{(n-1)^{n-1}} \\
         =& \frac{(m-1)^{m-1}}{m!} \lim_{n \rightarrow \infty} \frac{(n-m)^{n-m}}{(n-1)^{n-m}}
    \end{align*}
    and also
    \begin{align*}
        \lim_{n \rightarrow \infty} \frac{(n-m)^{n-m}}{(n-1)^{n-m}} = \lim_{n \rightarrow \infty} \left( 1 - \frac {m-1}{n-1} \right)^{n-m} = \lim_{n \rightarrow \infty} \left( 1 - \frac {m-1}{n-1} \right)^{n-1} = \frac{1}{e^{m-1}}
    \end{align*}
    which gives the desired expression.
\end{proof}

\begin{prop} \label{spike prop}
    Suppose that $m, n$ are integers with $2 \le  m < \sqrt{n}$. The following holds for $\bar \chi(n, r)$.
    
    (a) The following bounds hold:
    \begin{align*}
        & a_{m, n} \le \frac{\bar \chi(n, s_{m, n})}{n} \le M \le a_{m, n} + b_{m,n}
    \end{align*}
    where
    \begin{align*} 
        M =& \max \left\{ \frac 1n \bar \chi \left( n, \frac{1-r}{2} \right) \> \bigg| \>  r \in \left( \frac1{m+1} , \frac1m \right) \right\} \\
        s_{m, n} =& \frac {(m-1)n}{2(n-1)m} \\
        a_{m, n} =& \binom nm \frac{(m-1)^{m-1}(n-m)^{n-m}}{n (n-1)^{n-1}} \\
        b_{m, n} =& e n^{m-1} \left( 1 - \frac1{m+1} \right)^{n-1}
    \end{align*}
    (b) We have the following limits:
    \begin{align*}
        \lim_{n \rightarrow \infty} a_{m, n} = \frac{(m-1)^{m-1}}{m! e^{m-1}}, \quad \lim_{n \rightarrow \infty} b_{m, n} = 0
    \end{align*}
    (c) Suppose additionally that $n > 2m^2$. Then for each $\lambda \in [0, 1]$, we have that:
    \[ \left| r - \frac{n - m}{(n-1)m} \right| < \frac{(1-\lambda) \sqrt{(m-1)(n-m)} }{m(n-1)^{3/2}} \implies \frac1n \bar \chi \left( n, \frac{1-r}2 \right) > \lambda a_{m, n} \]
    This condition for $r$ in particular satisfies $r \in \left( \frac1{m+1}, \frac1m \right]$.
\end{prop}
\begin{proof}
    Let $r \in \left( \frac1{m+1}, \frac1m \right]$ and also write $r = \frac1m - t$, with $t \in \left[ 0, \frac1{m(m+1)} \right]$. Then we may rewrite the normalised expected Euler characteristic as follows:
    \begin{align*}
        \bar \chi\left(n, \frac{1-r}2 \right) =& \sum_{k=1}^m \binom nk (1-kr)^{k-1} (kr)^{n-k} \\
        =& \sum_{k=1}^m \binom nk \left( 1 - \frac km + kt \right)^{k-1} \left( \frac km - kt \right)^{n-k}
    \end{align*}
    We now claim that the $k=m$ term is the dominant one among the above summands. As such, we split the above sum as:
    \[ \bar \chi\left(n, \frac{1-r}2 \right) = f_{m,n}(t) + E \]
    where
    \begin{align*}
        f_{m, n}(t) =& \binom nm (mt)^{m-1}(1-mt)^{n-m},\\
        E =& \sum_{k=1}^{m-1} \binom nk \left( 1 - \frac km + kt \right)^{k-1} \left( \frac km - kt \right)^{n-k}
    \end{align*}
    Since $m < \sqrt{n}$, we have $s_{m,n} = \frac1{n-1}(1-\frac1m) < \frac1{m(m+1)}$. Therefore, the previous Lemma tells us that $f_{m, n}(t)$ achieves (global) maximum at $\tilde s \in \left(0, \frac1{m(m+1)}\right]$, with the maximum value given by:
    \[ f_{m, n}(\tilde s) = n \cdot a_{m, n} \text{, where }  a_{m,n} = \binom nm \frac{(m-1)^{m-1}(n-m)^{n-m}}{n(n-1)^{n-1}} \]
    We also bound $E$ as follows, using the inequality $\frac m{m+1} < 1-mt \le 1$:
    \begin{align*}
        E =& \sum_{k=1}^{m-1} \binom nk \left( 1 - \frac km (1-mt) \right)^{k-1} \left( \frac km (1-mt) \right)^{n-k} \\
        \le & \sum_{k=1}^{m-1} \binom nk \left( 1 - \frac1{m+1} \right)^{k-1} \left( 1 - \frac 1m \right)^{n-k} \\
        \le & \sum_{k=1}^{m-1} \frac{n^k}{k!} \left( 1 - \frac1{m+1} \right)^{n-1} \\
        \le & e n^{m-1} \left( 1 - \frac1{m+1} \right)^{n-1}
    \end{align*}
    This shows (a). Now (b) follows from the previous Lemma and the fact that $(1-\frac 1{m+1})^n$ term causes exponential decay for $b_{m, n}$. 
    
    (c) follows from (c) of the previous Lemma. We additionally impose the condition $n > 2m^2$, so that the endpoints of $t$ satisfying the condition fall in the interval $t \in \left[ 0, \frac1{m(m+1)} \right)$.
\end{proof}

The following yields Theorem A2.

\begin{prop}\label{main prop 1}
    Let $m\ge 2$, $\epsilon>0$. The following holds for sufficiently large $n$:
    \begin{align*}
         r \in \left[ \alpha^-, \alpha^+ \right] \implies \frac1n \bar \chi\left(n, \frac{1-r}2\right) \in \bigg[ (1-\epsilon) \omega_m, (1+\epsilon) \omega_m \bigg]
    \end{align*}
    where
    \begin{align*}
        \alpha^{\pm} =& \frac{n-m}{(n-1)m}\left( 1 \pm \frac{\epsilon \sqrt{m-1}}{n} \right), \quad \omega_m = \frac{(m-1)^{m-1}}{m! e^{m-1}}
    \end{align*}
\end{prop}
\begin{proof}
    This follows directly from the previous Proposition. $\alpha^{\pm}$ are slight relaxations of the interval in (c), where we set $\lambda = 1-\epsilon$:
    \begin{align*}
        & \left[ \frac{n-m}{(n-1)m} - \epsilon R_1 , \frac{n-m}{(n-1)m} + \epsilon R_1\right] \supseteq \left[ \frac{n-m}{(n-1)m} (1 - \epsilon R_2) , \frac{n-m}{(n-1)m} (1 + \epsilon R_2) \right] \\
        & \text{where } R_1 = \frac{\epsilon \sqrt{(m-1)(n-m)} }{m(n-1)^{3/2}}, \quad R_2 = \frac{\sqrt{m-1}}{n}
    \end{align*}
\end{proof}

\section{Random homotopy types}

\subsection{Constraints on homotopy types}

Let $\Unif_n = \{ i/n \>|\> i=0, 1, \ldots n-1 \} \subset \Circle$ be the set of $n$ equally spaced points. Let $\mathcal N(n, k)$ be the nerve complex on $\Unif_n$ defined by the open cover consisting of closed intervals $[i/n, (i+k)/n]$.

\begin{lem}
    We have that:
    \begin{align*}
        \Cech(\Unif_n, r) = \Cech \bigg( \Unif_n, \frac{\lfloor 2rn \rfloor }{2n} \bigg) = \mathcal N(n, \lfloor 2rn\rfloor)
    \end{align*}
\end{lem}

The following result is from \cite{ncca}:
\begin{prop}
    \begin{align*}
        \mathcal N(n,k) \simeq \begin{cases} 
        \vee^{n-k-1} \sphere^{2l} & \text{if } \frac kn = \frac l{l+1} \\
        \sphere^{2l+1} & \text{if } \frac kn \in \left( \frac l{l+1}, \frac{l+1}{l+2} \right)
        \end{cases}
    \end{align*}
    Note that if $(k, n) = (jl, j(l+1))$, then $n-k-1 = j-1$, so that $\vee^{n-k-1}\sphere^{2l} = \vee^{j-1} \sphere^{2l}$.
\end{prop}

Using the above, we easily show that:
\begin{prop}
    Given $r \in (0, 1/2)$, the following two subsets of $\mathbb{Z}^3$ are equal:
    \begin{align*}
        \bigg\{ (n,a,b) \>\bigg|\> \Cech(\Unif_n, r) \simeq \vee^a \sphere^{2b} \bigg\} = 
        \bigg\{ ((a+1)(b+1), a, b) \>\bigg|\> b+1 \le \rr^{-1} , \> a+1 \le \frac1{1-(b+1) \rr} \bigg\}
    \end{align*}
    where $\rr = 1-2r$. In particular, if $\rr^{-1} \in [k, k+1)$, then $b \in \{0, 1, 2, \ldots k-1\}$ and we have $a \le k-1$ when $b \le k-2$. 
\end{prop}
\begin{proof}
To have $\mc N(n, \lfloor 2rn \rfloor) = \Cech(\Unif_n, r) \simeq \vee^a \sphere^{2b}$, we see from the previous Proposition that the condition is given by $(\lfloor 2rn\rfloor , n) = ((a+1)b, (a+1)(b+1))$. This determines $n$ from $(a,b)$. The condition on $\lfloor 2rn\rfloor$ is then:
\begin{align*}
    & (a+1)b \le 2r(a+1)(b+1) < (a+1)b + 1 \\
    \iff & \rr(b+1) \le 1, \> a < \rr(a+1)(b+1) \\
    \iff & (b+1) \le \rr^{-1}, \> (a+1) < (1 - \rr (b+1))^{-1}
\end{align*}
as desired.
\end{proof}
\textbf{Remark.} At fixed $l$, let $k=b+k_0$. Then $\frac1{1-(b+1)/k} = 1 + \frac{b+1}{k_0-1}$ and changing $k_0$ by a single value can have a heavy effect on the upper bound.

\begin{prop}\label{homotopy type constraint}
    Let $r \in (0, 1/2)$ and $n$ be given; define $\tilde r = 1-2r$ and let $k = \lfloor \tilde{r}^{-1}\rfloor$. We have the following relations between subsets of $\mathbb{Z}^2$:
    \begin{align*}
        & \bigg\{ (a,b) \>\bigg|\> \Cech(\mathbf{Y} , r) \simeq \vee^a \sphere^{2b}, \> \mathbf{Y} \subset \sphere^1, \> \# \mathbf{Y} = n \bigg\} \\
        =& \bigg\{ (a,b) \>\bigg|\> \Cech(\Unif_m , r) \simeq \vee^a \sphere^{2b}, \> m \le n \bigg\} \\
        \subseteq & \bigg\{ (a,b) \>\bigg|\> b+1 \le k, \> a+1 \le \min\bigg( \frac n{b+1}, \frac1{1-(b+1)\tilde r}\bigg) \bigg\} \\
        \subseteq & \bigg\{ (a,b) \>\bigg|\> b+1 \le k-1, \> a+1 \le \frac k{k-b-1} \bigg\} \cup \bigg\{ (a, k-1) \>\bigg|\> a+1 \le \frac nk \bigg\}
    \end{align*}
    where in the final expression, $k/0 = \infty$ by convention.
\end{prop}
\begin{proof}
    The first equality holds because for every $\mathbf{Y} \subset \sphere^1$, there exists $\mathbf{Y}' \subset \mathbf{Y}$ such that $\Cech(\mathbf Y, r) \simeq \Cech(\mathbf Y', r) \simeq \Cech(\Unif_m, r)$, where $m = \# \mathbf{Y}'$ \cite{ncca}. The first inclusion follows from the previous Proposition. The second inclusion follows from separating the two cases $b+1 < k$ and $b+1 = k$.
\end{proof}

\subsection{Probabilistic bounds}

For a topological space $K$, we define the following notation for probability:
\[ p(K, n, r) = \Prob[\Cech(\mathbf{X}_n, r) \simeq K] \]

We generally have the following:
\[ \bar\chi(n, r) = \EE[ \chi(\Cech(\mathbf X_n, r)) ] = \sum_{K} \chi(K) \cdot p(K, n, r) \]
where the sum is well-defined because there are only finitely many combinatorial structures that $\Cech(\mathbf X_n, r)$ can take. Furthermore if we let $k = \lfloor (1-2r)^{-1} \rfloor$, then Proposition \ref{homotopy type constraint} tells us that:
\[ \{ K \>|\> p(K, n, r) > 0 \} \subseteq \bigg\{ \vee^a \sphere^{2b} \>\bigg|\> b+1 \le k-1, a+1 \le \frac{k}{k-b-1} \bigg\} \cup \bigg\{ \vee^a \sphere^{2k-2} \>\bigg|\> a+1 \le \frac nk \bigg\} \]
From this we infer that\footnote{By convention, in the summation we only consider $a \le 0$ when $b=0$ and instead consider $a>0$ when $b>0$. This is so that the singleton set $\vee^a \sphere^{2b} = *$ is counted only once.}:
\begin{align*}
    \bar\chi(n, r) =& A_{<k} + A_{k} \\
    \text{where } A_{<k} =& \sum_{\substack{0 \le b \le k-2 \\ (a+1)(k-b-1) \le k}} (a+1) \cdot p(\vee^a \sphere^{2b}, n, r) \\
    A_{k} =& \sum_{1 < a+1 \le n/k} (a+1) \cdot p(\vee^a \sphere^{2k-2}, n, r)
\end{align*}
where we used $\chi(\vee^a \sphere^{2b}) = a+1$. Since sum of probabilities is $1$, applying the constraint $(a+1)(k-b-1) \le k$ implies that $A_{<k} \le k$. This implies the following:
\begin{prop}\label{main prop 2}
    The following holds:
    \begin{align*}
        A_{k} \le \bar \chi(n, r) \le k + A_{k}
    \end{align*}
    where
    \[ A_{k} = \sum_{1 < a+1 \le n/k} (a+1) \cdot p(\vee^a \sphere^{2k-2}, n, r) \]
\end{prop}

\begin{cor}[Theorem A3]
    Let $k \ge 2$. Given $\epsilon > 0$, the following hold for sufficiently large $n$:
    \[ 1-2r \in \left( \frac1{k+1}, \frac 1k \right] \implies \frac{\bar \chi(n, r)}n - \epsilon \le  \frac{\bar b_{2k-2}(n, r)}n \le \frac{\bar \chi(n, r)}n \]
\end{cor}

Now we're interested in controlling probabilities that $\vee^a \sphere^{2k-2}$ appear, with large $n$. For this, we further define following:
\begin{align*}
    p_a =& p(\vee^a \sphere^{2k-2}, n, r) \\ 
    l =& \lfloor n/k \rfloor - 1 \\
    \tilde\delta =& \lceil \delta n/ k \rceil - 1 \\
    A_{k, \delta} :=& \sum_{\tilde\delta \le a \le l} (a+1) \cdot p_a = \sum_{\delta n/k \le a+1 \le n/k} (a+1) \cdot p_a \\
    B_{k, \delta} :=& \sum_{\tilde \delta \le a \le l} p_a
\end{align*}
To produce bounds for $B_{k, \delta}$, we split $A_k$ into two parts:
\begin{align*}
    A_k = \bigg( 2p_1 + 3p_2 + \cdots + \tilde\delta p_{\tilde\delta - 1} \bigg) + \bigg( (\tilde\delta + 1) p_{\tilde\delta} + \cdots + (l+1) p_l \bigg)
\end{align*}
from which it directly follows that:
\begin{align*}
    (\tilde\delta + 1) B_{k,\delta} \le A_k \le \tilde\delta (1-B_{k, \delta}) + (l+1) B_{k, \delta}
\end{align*}
and therefore
\begin{align*}
    \implies & (\tilde\delta + 1) B_{k,\delta} \le A_k \le \tilde\delta + (l+1-\tilde\delta) B_{k,\delta} \\
    \implies & \frac{A_k - \tilde\delta}{l+1- \tilde\delta} \le B_{k,\delta} \le \frac{A_k}{\tilde\delta + 1} \\
    \implies & \frac{A_k - \lceil \delta n / k \rceil + 1}{\lfloor n/k \rfloor - \lceil \delta n / k \rceil + 1} \le B_{k,\delta} \le \frac{A_k}{\lceil \delta n / k \rceil} \\
    \implies & \frac{A_k - \delta n / k}{(1-\delta) (n / k) + 1} \le B_{k,\delta} \le \frac{A_k}{\delta n/k}
\end{align*}
In summary, we have the following:

\begin{prop}\label{main prop 3}
    Let $n \in \ZZ^+$, $\delta \in (0,1)$, $r \in (0,1/2)$ be given, and let $k = \lfloor (1-2r)^{-1} \rfloor$. The following holds:
    \[ \frac{k A_k - \delta n}{(1-\delta)n + k} \le B_{k,\delta} \le \frac{k A_k}{\delta n} \]
    where
    \begin{align*}
        A_{k} =& \sum_{1 < a+1 \le n/k} (a+1) p_a, \quad B_{k, \delta} := \sum_{\delta n /k \le a+1 \le n/k} p_a, \quad p_a := p(\vee^a \sphere^{2k-2}, n, r)
    \end{align*}
\end{prop}

Now Propositions \ref{main prop 1}, \ref{main prop 2}, \ref{main prop 3} imply the following, which is a more general version of Theorem C:
\begin{thm}\label{elder C}
    Let $r \in [\frac14, \frac12)$ and let $k = \lfloor (1-2r)^{-1} \rfloor$. Given $\epsilon, \delta \in (0,1)$, the following implication holds for large enough $n$:
    \begin{align*}
        1 - 2r \in [\alpha^-, \alpha^+] \implies 
        B_{k, \delta} \in [\beta^- - \epsilon, \beta^+ + \epsilon]
    \end{align*}
    where
    \begin{align*}
        \alpha^{\pm} =& \frac1k \frac{n-k}{n-1} \bigg( 1 \pm \frac{ \sqrt{k-1}}{n} \cdot \frac{\delta(1-\delta)}5 \cdot \epsilon \bigg), \\
        \beta^- =& \frac{k \omega_k - \delta}{1-\delta}, \quad \beta^+ = \frac{k\omega_k}{\delta} \\
        \omega_k =& \frac{(k-1)^{k-1}}{k! e^{k-1}} \\
        B_{k, \delta} :=& \sum_{\delta n / k \le a+1 \le n / k} p(\vee^a \sphere^{2k-2}, n, r)
    \end{align*}
    The bounds $\beta^{\pm}$ satisfy $\beta^- \le k\omega_k \le \beta^+$. Also $\beta^- > 0$ iff $\delta < k\omega_k$ and $\beta^+ < 1$ iff $\delta > k\omega_k$. 
\end{thm}
\begin{proof}
We first describe the heuristic reasoning for the bounds, which is rather simple. Proposition \ref{main prop 3} gives us:
\[ \frac{kA_k - \delta n}{(1-\delta)n + k} \le B \le \frac{kA_k}{\delta n} \]
By Proposition \ref{main prop 1} and \ref{main prop 2}, the upper bound has the following approximations:
\[ \frac{kA_k}{\delta n} \approx \frac{k\bar\chi}{\delta n} \approx \frac{k\omega_k}{\delta} \]
and similarly the lower bound has the following approximations:
\[ \frac{kA_k - \delta n}{(1-\delta)n + k} \approx \frac{kA_k - \delta n}{(1-\delta)n} \approx \frac{k\bar\chi - \delta}{1-\delta} \approx \frac{k\omega_k - \delta}{1-\delta} \]
The actual proof becomes more complicated due to using a different choice of $\epsilon$ in applying Proposition \ref{main prop 1}. 

Let $\epsilon' = \delta(1-\delta) \cdot \epsilon/5$. We apply Proposition \ref{main prop 1} with $\epsilon'$ taking the role of $\epsilon$, and this gives the choice of $\alpha^{\pm}$ in the theorem. Therefore $r \in [\alpha^-, \alpha^+]$ implies the following:
\begin{align}
    (1-\epsilon')\omega_k \le \frac{\bar\chi}n \le (1+\epsilon')\omega_k \label{omega chi bound}
\end{align}
Before going further, we note the following inequalities for $\epsilon'$, which we will use later:
\begin{align}
    & \epsilon' = \frac{\delta(1-\delta) \epsilon}{4+1} \le \frac{\delta(1-\delta)\epsilon}{4 + \delta(1-\delta)\epsilon} \nonumber \\
    \implies & \frac{\epsilon'}{1-\epsilon'} \le \frac{\delta(1-\delta) \epsilon}4 \nonumber \\
    \implies & \frac{\epsilon'}{1-\epsilon'} \le \min\bigg( 4\delta, \delta^{-1}-1, 1 \bigg) \cdot \frac{\epsilon}4 \label{epsilon prime}
\end{align}

\textbf{Upper bound.}

By Equation \eqref{omega chi bound} and Proposition \ref{main prop 2}, we have:
\begin{align*}
    \frac{k\omega_k}{\delta} \ge \frac1{1+\epsilon'} \frac{k\bar\chi}{\delta n} \ge \frac1{1+\epsilon'} \frac{kA_k}{\delta n}
\end{align*}
By Equation \eqref{epsilon prime}, we have that:
\[ \frac1{1+\epsilon'} \frac{kA_k}{\delta n} \ge \frac{kA_k}{\delta n} - \epsilon \]
Then Proposition \ref{main prop 3} applies and we have the upper bound.

\textbf{Lower bound.}

By Equation \eqref{omega chi bound} and Proposition \ref{main prop 2}, we have:
\[ \frac{k\omega_k - \delta}{1-\delta} \le \frac1{1-\delta} \bigg( \frac1{1-\epsilon'} \frac{k\bar\chi}n - \delta\bigg) \le \frac1{1-\delta}\bigg( \frac1{1-\epsilon'} \frac{k^2+kA_k}n - \delta \bigg) \]
Let $L_0$ be the right hand side. We rewrite it as follows:
\begin{align*}
    L_0 = L_1 + E_1 = L_2 + E_1 + E_2
\end{align*}
where
\begin{align*}
    L_1 =& \frac{kA_k - \delta n }{(1-\delta)(1-\epsilon')n}, \> E_1 = \frac{\delta \epsilon' + k^2/n}{(1-\delta)(1-\epsilon')} \\
    L_2 =& \frac{kA_k - \delta n}{(1-\delta)n + k}, \> E_2 = \frac{kA_k - \delta n}{(1-\delta)(1-\epsilon')n} \cdot \frac{k + (1-\delta)n \epsilon'}{(1-\delta)n + k}
\end{align*}
By Equation \eqref{epsilon prime}, the relation $kA_k \le n$ and by taking $n$ large enough, we see that
\[ E_1, E_2 \le \epsilon/2 \]
This implies that:
\[ \frac{k\omega_k - \delta}{1-\delta} - \epsilon \le L_0 - \epsilon = L_2 + E_1 + E_2 - \epsilon \le L_2 \]
Then again Proposition \ref{main prop 3} applies and we have the lower bound.

\end{proof}

We remark that Theorem C is obtained by setting $\epsilon = \delta = (1-\alpha)k\omega_k / 2$. The gap $\alpha^+ - \alpha^-$ is replaced by a smaller but simpler quantity.

\section{Odd spheres}

We prove Theorem B using the stability of persistence diagram. In this case, we will be using the Cech complex constructed from the full set of the circle, and then bound the Gromov-Hausdorff distance between the full circle and a finite sample of it. We use the following result from \cite{vr_circle}:
\begin{thm}
    The homotopy types of the Rips and Cech complexes on the circle of unit circumference are as follows:
    \begin{align*}
        \on{VR}(\SS^1, r) &\simeq \begin{cases} \SS^{2l+1} & \text{,if }  \frac{l}{2l+1} < r < \frac{l+1}{2l+3} \\ \bigvee^{\mathfrak c} \SS^{2l} & \text{,if } r = \frac{l}{2l+1} \end{cases} \\
        \Cech(\SS^1, r) &\simeq \begin{cases} \SS^{2l+1} & \text{,if }  \frac{l}{2l+2} < r < \frac{l+1}{2l+4} \\ \bigvee^{\mathfrak c} \SS^{2l} & \text{,if } r = \frac{l}{2l+2} \end{cases}
    \end{align*}
    where $\mathfrak c$ is the cardinality of the continuum.
\end{thm}

We also note the stability of persistence:
\begin{thm}[Stability of Persistence]
    If $X, Y$ are metric spaces and $\mc D_k M$ is the $k$-dimensional persistence diagram of persistence module $M$, then
    \begin{align*}
        \d_B(\mc D_k \mathbf{VR}(X), \mc D_k \mathbf{VR}(Y)) &\le \d_{GH}(X, Y) \\
        \d_B(\mc D_k \Cech (X), \mc D_k \Cech(Y)) &\le \d_{GH}(X, Y)
    \end{align*}
    where $\d_{GH}$ denotes the Gromov-Hausdorff distance. 
\end{thm}

The following proposition is a more precise version of Theorem B, which specifies an explicit lower bound for the probabilities of homotopy equivalence:

\begin{prop}
    For each $l \ge 0$ and $t \in (\frac l{2l+2}, \frac{l+1}{2l+4})$, the following holds with probability at least $Q_n(r'/2)$:
    \[ \Cech(\mathbf X_n, t) \simeq \SS^{2l+1} \]
    where $r'$ is:
    \[ r' = \frac1{4(l+1)(l+2)} - \left| t - \frac{2l^2+4l+1}{4(l+1)(l+2)} \right| \]
\end{prop}
\begin{proof}
	
Consider a random sample $\mathbf X_n = (X_1, \ldots X_n)$. Then with probability $Q_n(r/2)$, arcs of radius $r$ centered at $\mathbf X_n$ covers $\SS^1$, so that $\d_{GH}(\mathbf X_n, \SS^1) \le \d_H(\mathbf X_n, \SS^1) \le r$. This implies:
\begin{align*}
    \d_B(\mathcal D_k \Cech(\mathbf X_n) , \mathcal D_k \Cech(\SS^1)) \le \d_{GH}(\mathbf X_n, \SS^1) \le r
\end{align*}
For each $l\ge 0$, we have that:
\[ \mathcal D_{2l+1} \Cech(\SS^1) = \left \{ \left ( \frac{l}{2l+2}, \frac{l+1}{2l+4} \right) \right \} \]
so that the definition of the bottleneck distance implies that 
\begin{align*}
    & \exists (u, v) \in \mathcal D_{2l+1} \Cech(\mathbf X_n) \\
    \text{with } & \frac l{2l+2} - r \le u \le \frac l{2l+2} + r \\
    & \frac {l+1}{2l+4} - r \le v \le \frac {l+1}{2l+4} + r 
\end{align*}
This implies that whenever $\frac{l}{2l+2} + r \le t \le \frac{l+1}{2l+4} - r$, we have:
\[ 1 \le \dim H_{2l+1} \Cech(\mathbf X_n, t) \]
and due to the enumeration of possible homotopy types, we have that:
\[ \Cech(\mathbf X_n, t) \simeq \SS^{2l+1} \]
The condition translates to $\left | t - \frac12 \left( \frac{l}{2l+2} + \frac{l+1}{2l+4} \right) \right | < \frac12 \left( \frac{l+1}{2l+4} - \frac{l}{2l+2} \right) - r$, or equivalently
\[ \left| t - \frac{2l^2+4l+1}{4(l+1)(l+2)} \right| < \frac1{4(l+1)(l+2)} - r \]
and thus we obtain the proof.

\end{proof}

\bibliographystyle{abbrv}
\bibliography{refs}

\end{document}